\documentclass[reqno]{amsart}
\usepackage{mathrsfs}
\usepackage{color}
\usepackage{amsmath}
\usepackage{amsfonts}
\usepackage{amssymb}
\usepackage{graphicx}
\usepackage{hyperref}

\newtheorem{theorem}{Theorem}[section]
\newtheorem{corollary}[theorem]{Corollary}
\newtheorem{lemma}[theorem]{Lemma}
\newtheorem{proposition}[theorem]{Proposition}

\newtheorem{definition}[theorem]{Definition}
\newtheorem{question}[theorem]{Question}

\newtheorem{remark}[theorem]{Remark}

\numberwithin{equation}{section}

\begin{document}

\title[Equivalence of the sharp effectiveness results of SOP]
 {Equivalence of the sharp effectiveness results of strong openness property}

\author{Shijie Bao}
\address{Shijie Bao: Institute of Mathematics, Academy of Mathematics and Systems Science, Chinese Academy of Sciences, Beijing 100190, China.}
\email{bsjie@amss.ac.cn}

\author{Qi'an Guan}
\address{Qi'an Guan: School of
Mathematical Sciences, Peking University, Beijing 100871, China.}
\email{guanqian@math.pku.edu.cn}

\thanks{}



\date{}

\dedicatory{}

\commby{}


\begin{abstract}
    In this paper, we show the equivalence of the sharp effectiveness results of the strong openness property of multiplier ideal sheaves obtained in \cite{BG1} using $\xi-$Bergman kernels and in \cite{Guan19} using minimal $L^2$ integrals.
\end{abstract}

\maketitle

\section{Introduction}\label{Introduction}

There are two distinct sharp effectiveness results of the strong openness property of multiplier ideal sheaves, proven in \cite{Guan19} by Guan and in \cite{BG1} by Bao-Guan, using minimal $L^2$ integrals and $\xi$-Bergman kernels respectively. In the present paper, we will prove that these two effectiveness results are equivalent.

\subsection{Notations and conventions}

Before we discuss the background and motivation of the present paper, we recall some notations and conventions.

\subsubsection{$\xi-$Bergman kernel}
We recall the $\xi-$Bergman kernel (a version of generalized Bergman kernel) defined in \cite{BG1} in the following.

Denote
\begin{equation*}
    \ell_1^{(n)}:=\Big\{(\xi_{\alpha})_{\alpha\in\mathbb{N}^n} \colon \xi_{\alpha}\in\mathbb{C}, \  \sum_{\alpha\in\mathbb{N}^n}|\xi_{\alpha}|\rho^{|\alpha|}<+\infty, \ \forall \rho>0\Big\}.
\end{equation*}
We also denote
\[\ell_0^{(n)}:=\Big\{(\xi_{\alpha})_{\alpha\in\mathbb{N}^n} \colon \xi_{\alpha}\in\mathbb{C}, \  \exists k\in\mathbb{N}, \ \text{s.t.} \ \xi_{\alpha}=0, \ \forall\alpha \ \text{with} \  |\alpha|>k\Big\}\subseteq \ell_1^{(n)},\]
in the present paper. $\ell_1^{(n)}$ can be seen as $\mathcal{O}(\mathbb{C}^n)$, the space of entire functions on $\mathbb{C}^n$, while $\ell^{(n)}_0$ can be seen as $\mathbb{C}[z_1,\ldots,z_n]$, the space of polynomials on $\mathbb{C}^n$.

Let $D$ be a domain in $\mathbb{C}^n$. For any $z_0\in D$, $\xi=(\xi_{\alpha})\in\ell_1^{(n)}$, and $f(z)=\sum_{\alpha\in\mathbb{N}^n}c_{\alpha}(z-z_0)^{\alpha}\in\mathcal{O}_{z_0}$, denote
\begin{equation*}
    (\xi\cdot f)(z_0):=\sum_{\alpha\in\mathbb{N}^n}\xi_{\alpha}\frac{(D^{\alpha}f)(z_0)}{\alpha!}=\sum_{\alpha\in\mathbb{N}^n}\xi_{\alpha}c_{\alpha}.
\end{equation*}
\begin{definition}[see \cite{BG1}]
    Let $\xi\in\ell_1^{(n)}$. The Bergman kernel with respect to $\xi$ on $D$ is defined as:
\begin{equation*}
    K_{\xi,D}(z):=\sup\big\{|(\xi\cdot f)(z)|^2 \colon f\in A^2(D), \ \|f\|_D\le 1\big\},
\end{equation*}
where $A^2(D)=L^2(D)\cap\mathcal{O}(D)$ is the Bergman space, and $\|f\|^2_D=\int_D|f|^2$.
\end{definition}
It can be checked that $K_{\xi,D}(z)<+\infty$ (see \cite[Lemma 2.3]{BG1}). 

Suppose $I$ is an ideal of $\mathcal{O}_o$, the ring of holomorhpic function germs on the origin $o\in\mathbb{C}^n$. Denote
\[\tilde{\ell}_I:=\big\{\xi\in\ell_1^{(n)} : (\xi\cdot f)(o)=0, \ \forall (f,o)\in I\big\},\]
and $\ell_I:=\tilde{\ell}_I\setminus\{0\}$. $\tilde{\ell}_I$ is a linear subspace of $\ell^{(n)}_1$, where it is not trivial when $I$ is a proper ideal of $\mathcal{O}_o$.

Now let $D$ be a domain in $\mathbb{C}^n$ containing the origin $o$, $I$ a proper ideal of $\mathcal{O}_o$, and $(F,o)\in \mathcal{O}_o$. Recall the notation
\[B(F,I,D):=\sup_{\xi\in\ell_I}\frac{|(\xi\cdot F)(o)|^2}{K_{\xi,D}(o)}.\]
\begin{proposition}[Proposition 4.6 in \cite{BG1}]\label{prop-separation}
    $B(F,I,D)=0$ if and only if $(F,o)\in I$.
\end{proposition}
This proposition was proved in \cite{BG1} by the separation theorem in the functional analysis theory.

In the present paper, we also denote
\[B^{\circ}(F,I,D):=\sup_{\xi\in\ell_I\cap\ell_0^{(n)}}\frac{|(\xi\cdot F)(o)|^2}{K_{\xi,D}(o)}.\]
Clearly, we have $B^{\circ}(F,I,D)\le B(F,I,D)$.

\subsubsection{Minimal $L^2$ integral}
The second author of the present paper introduced the minimal $L^2$ integrals in \cite{Guan19} to establish a sharp effectiveness result of the strong openness property. We recall the definition and notation of the minimal $L^2$ integral below.

Again let $D$ be a domain in $\mathbb{C}^n$ containing the origin $o$, $I$ a proper ideal of $\mathcal{O}_o$, and $(F,o)\notin I$. Recall the notation of \emph{minimal $L^2$ integral} (\cite{Guan19})
\[C_{F,I}(D):=\inf\big\{\|\tilde{F}\|_D^2 : \tilde{F}\in A^2(D), \ (\tilde{F}-F,o)\in I\big\},\]
and if the set in the right hand is empty, then denote $C_{F,I}(D)=+\infty$. By Montel's theorem, if $C_{F,I}(D)<+\infty$, there exists a unique $G\in A^2(D)$ with $\|G\|_{D}^2=C_{F,I}(D)$ and $(G-F,o)\in I$. One can also note that $C_{F,I}(D)=0$ if and only if $(F,o)\in I$ (see \cite{Guan19}).

Denote
\[A^2(D,I):=\big\{f\in A^2(D) : (f,o)\in I\big\},\]
which is a closed subspace of $A^2(D)$. Then there exists a unique closed subspace $A^2(D,I)^{\perp}$ of $A^2(D)$, which is the orthogonal complement of $A^2(D,I)$ in $A^2(D)$. Moreover, for any $F\in A^2(D)\setminus A^2(D,I)$, the unique $G\in A^2(D)$ with $\|G\|_{D}^2=C_{F,I}(D)$ and $(G-F,o)\in I$ is actually the image of $F$ under the orthogonal projection from $A^2(D)$ to $A^2(D,I)^{\perp}$.

With simple computations, one can check that $B(F,I,D)\le C_{F,I}(D)$ holds. More precisely, assume that $\tilde{F}\in A^2(D)$ satisfying $(\tilde{F}-F,o)\in I$ and $\|\tilde{F}\|_D^2=C_{F,I}(D)<+\infty$. Then $(\xi\cdot F)(o)=(\xi\cdot \tilde{F})(o)$ for any $\xi\in\ell_I$, and thus
\[B(F,I,D)=\sup_{\xi\in\ell_I}\frac{|(\xi\cdot F)(o)|^2}{K_{\xi,D}(o)}\le\sup_{\substack{\xi\in\ell_I, \\ (\xi\cdot F)(o)\neq 0}}\frac{|(\xi\cdot F)(o)|^2}{|(\xi\cdot \tilde{F})(o)|^2/\|\tilde{F}\|_D^2}=\|\tilde{F}\|_D^2=C_{F,I}(D).\]

\subsection{Background and motivation}
Let $D$ be a pseudoconvex domain in $\mathbb{C}^n$ containing the origin $o$, $\varphi$ a negative plurisubharmonic function on $D$ with $\varphi(o)=-\infty$, and $F$ a holomorphic function on $D$.

Recall the \emph{jumping number}
\[c_o^F(\varphi):=\sup\big\{c\ge 0 : (F,o)\in\mathcal{I}(c\varphi)_o\big\},\]
and the \emph{multiplier ideal}
\[\mathcal{I}(c\varphi)_o:=\big\{(G,o)\in\mathcal{O}_o : |G|^2e^{-c\varphi} \ \text{is\ } L^1 \ \text{integrable near\ } o\big\}, \ \forall c>0.\]

The \emph{strong openness property} of multiplier ideal sheaves (\cite{GZ15a}) shows that: \emph{if $\int_D|F|^2e^{-\varphi}<+\infty$, then there must exist some $p>1$ such that $(F,o)\in\mathcal{I}(p\varphi)_o$, or equivalently, $(F,o)\notin\mathcal{I}(c_o^F(\varphi)\varphi)_o$}. The effectiveness of the strong openness property was established in \cite{GZ15b}. Moreover, \cite{Guan19} gave the following sharp effectiveness result.
\begin{theorem}[\cite{Guan19}]\label{thm-effect.Guan}
    Suppose $\int_D|F|^2e^{-\varphi}<+\infty$. Then for any $p>1$ satisfying
    \[\frac{p}{p-1}>\frac{\int_D|F|^2e^{-\varphi}}{C_{F,\mathcal{I}_+(c_o^F(\varphi)\varphi)_o}(D)},\]
    we have $(F,o)\in\mathcal{I}(p\varphi)_o$.
\end{theorem}

Afterwards, \cite{BG1} proved the following result, which gave the optimal $L^2$ extension approach to the effectiveness result of the strong openness property.

\begin{theorem}[\cite{BG1}]\label{thm-effect.BG}
    Suppose $\int_D|F|^2e^{-\varphi}<+\infty$. Then for any $p>1$ satisfying
    \[\frac{p}{p-1}>\frac{\int_D|F|^2e^{-\varphi}}{B(F,\mathcal{I}_+(c_o^F(\varphi)\varphi)_o, D)},\]
    we have $(F,o)\in\mathcal{I}(p\varphi)_o$.
\end{theorem}

In \cite{BG1}, Theorem \ref{thm-effect.BG} was proved for bounded pseudoconvex domains  $D$, but it is easy to generalize the result to all pseudoconvex domains.

These two effectiveness results of the strong openness property were both shown to be sharp in \cite{Guan19} and \cite{BG1}, respectively. Theorem \ref{thm-effect.Guan} was established by estimating the minimal $L^2$ integrals on the sublevel sets of plurisubharmonic functions in \cite{Guan19}, while Theorem \ref{thm-effect.BG} was proved using the log-plurisubharmonicity of fiberwise $\xi$-Bergman kernels in \cite{BG1}. There are differences between minimal $L^2$ integrals and $\xi$-Bergman kernels. For example, in \cite{BGY23}, Bao-Guan-Yuan showed that the log-convexity of the minimal  $L^2$  integrals on the sublevel sets of plurisubharmonic functions does not hold generally, but this property is true for $\xi$-Bergman kernels. This leads naturally to the following question:

\begin{question}\label{ques-equiv}
Can one show that effectiveness results of the strong openness property in Theorem \ref{thm-effect.BG} and Theorem \ref{thm-effect.Guan} are equivalent, completing the optimal $L^2$ extension approach to Theorem \ref{thm-effect.Guan}? Moreover, for any ideal $I$ of $\mathcal{O}_o$, do we have
\[B(F,I,D)=C_{F,I}(D)?\]
\end{question}

We provide an affirmative answer to Question \ref{ques-equiv} in the present paper, which completes the optimal $L^2$ extension approach to Theorem \ref{thm-effect.Guan}.

Additionally, the proof of Proposition \ref{prop-separation} (a key step in Theorem \ref{thm-effect.BG}) in \cite{BG1} relies on the separation theorem from functional analysis, which only guarantees the existence of a desired functional $\xi$ without providing a method to find it. Therefore, it is valuable to find a proof that does not rely on the separation theorem. In this paper, we will present a method for finding the functional in the proof of the main results.

\subsection{Main results and applications}

Now we show the main results of the present paper, with several applications.

\subsubsection{Main results}

The following theorem is the main theorem of the present paper, which provides the affirmative answer to Question \ref{ques-equiv}.

\begin{theorem}\label{thm-BFID.eq.CFID}
    Let $D$ be a domain in $\mathbb{C}^n$ containing the origin $o$, $I$ a proper ideal of $\mathcal{O}_o$, and $(F,o)\in\mathcal{O}_o$. Then
    \[B^{\circ}(F,I,D)=B(F,I,D)=C_{F,I}(D).\]
\end{theorem}

This theorem indicates that Theorem \ref{thm-effect.BG} (the effectiveness result of the strong openness property obtained in \cite{BG1}) is equivalent to Theorem \ref{thm-effect.Guan} (obtained in \cite{Guan19}).

\begin{corollary}
    Theorem \ref{thm-effect.BG} is equivalent to Theorem \ref{thm-effect.Guan}.
\end{corollary}

\subsubsection{Applications of the main results}
For any $\xi\in\ell_1^{(n)}$, the functional
\begin{flalign*}
    \begin{split}
        A^2(D) &\longrightarrow \mathbb{C}\\
        f &\longmapsto (\xi\cdot f)(o),
    \end{split}
\end{flalign*}
is linear and continuous. By Riesz's representation theorem, there exists some $T(\xi)\in A^2(D)$ such that $\langle f,T(\xi)\rangle=(\xi\cdot f)(o)$ for any $f\in A^2(D)$, where $\langle \cdot, \cdot \rangle$ is the inner product on the Hilbert space $A^2(D)$. It is known that $T(\ell_0^{(n)})$ is dense in $A^2(D)$ (see \cite[Appendix]{BG3}).

Theorem \ref{thm-BFID.eq.CFID} implies

\begin{corollary}\label{cor-dense}
    For any ideal $I$ of $\mathcal{O}_o$, $T\big(\tilde{\ell}_I\cap\ell_0^{(n)}\big)$ is dense in $A^2(D,I)^{\perp}$. 
\end{corollary}

Let $D$ be a bounded pseudoconvex domain in $\mathbb{C}^n$ containing the origin $o$, $\varphi$ a negative plurisubharmonic function on $D$ with $\varphi(o)=-\infty$, and $(F,o)\in\mathcal{O}_o$.

For any $\xi\in\ell_1^{(n)}$, the following limit
\[\gamma_{\xi}(\varphi):=\lim_{t\to +\infty}\frac{\log K_{\xi,\{\varphi<-t\}\cap D}(o)}{t}\in [0,+\infty]\cup\{-\infty\},\]
is called the \emph{$\xi-$cse} (a generalization of complex singularity exponent, see \cite{BGY22}) of $\varphi$, where it is proved in \cite{BGY22} that the limit always exists (since $\log K_{\xi,\{\varphi<-t\}\cap D}(o)$ is convex in $t$), and independent of the choice of the bounded pseudoconvex domain $D$ by the fact: if $\xi\neq 0$,
\begin{equation}\label{eq-gamma.xi.varphi.equals.inf}
    \gamma_{\xi}(\varphi)=\inf\big\{c\ge 0 : \xi\in\ell_{\mathcal{I}(c\varphi)_o}\big\}.
\end{equation}

Theorem \ref{thm-BFID.eq.CFID} also implies
\begin{corollary}\label{cor-jumping.number.min.xi.cse}
    If $c_o^F(\varphi)<+\infty$, then
    \[c_o^F(\varphi)=\min_{\xi\in\ell_0^{(n)}, \ (\xi\cdot F)(o)\neq 0}\gamma_{\xi}(\varphi).\]
\end{corollary}

Corollary \ref{cor-jumping.number.min.xi.cse} generalizes \cite[Corollary 3.2]{BGY22}.

\section{Preliminaries}

In this section, we recall some preliminaries which will be used in the proofs.

\subsection{Krull's lemma}

The famous Krull's lemma for Noetherian local rings is the key ingredient of the proof of Theorem \ref{thm-BFID.eq.CFID}.

\begin{lemma}[Krull's lemma, see \cite{AM}]\label{lem-Krull.lem}
    Let $R$ be a Noetherian local ring with the unique maximal ideal $\mathfrak{m}$. Then for any ideal $I$ of $R$,
    \[\bigcap_{k\in\mathbb{N}_+}(I+\mathfrak{m}^k)=I.\]
\end{lemma}

It is well-known that the ring $\mathcal{O}_o$ is a Noetherian local ring, with the unique maximal ideal $\mathfrak{m}=(z_1,\ldots,z_n)$.

\subsection{Orthonormal basis of the Bergman space}
First, we give an order for $\mathbb{N}^n$.

Let $\alpha,\beta\in\mathbb{N}^n$ be two multi-indices, where $\alpha=(\alpha_1,\ldots,\alpha_n)$, $\beta=(\beta_1,\ldots,\beta_n)$. Denote $|\alpha|:=\alpha_1+\cdots+\alpha_n$ (and the same for $|\beta|$). We write $\alpha\prec\beta$, if and only if

(1) $|\alpha|<|\beta|$; or

(2) $|\alpha|=|\beta|$, and there exists $j\in \{1,\ldots,n\}$ such that $\alpha_n=\beta_n$, $\ldots$, $\alpha_{j+1}=\beta_{j+1}$, and $\alpha_{j}<\beta_{j}$.

Let $D\subseteq\mathbb{C}^n$ be a domain with $o\in D$, and $\phi$ a plurisubharmonic function on $D$. Denote
\[A_{\phi}^2(D):=\left\{f\in\mathcal{O}(D)\colon \|f\|_{D,\phi}^2<+\infty\right\},\]
where $\|f\|_{D,\phi}^2:=\int_D|f|^2e^{-\phi}$.

The following lemma shows that we can have an orthonormal basis of $A_{\phi}^2(D)$ satisfying some specific conditions.

For any $\xi\in\ell_1^{(n)}$, the functional
\begin{flalign*}
    \begin{split}
        A_{\phi}^2(D) &\longrightarrow \mathbb{C}\\
        f &\longmapsto (\xi\cdot f)(o),
    \end{split}
\end{flalign*}
is linear and continuous. Thus, by Riesz representation theorem, there exists some $T_{\phi}(\xi)\in A_{\phi}^2(D)$ such that $\langle f,T_{\phi}(\xi)\rangle_{\phi}=(\xi\cdot f)(o)$ for any $f\in A_{\phi}^2(D)$, where $\langle \cdot, \cdot \rangle_{\phi}$ is the inner product on the Hilbert space $A_{\phi}^2(D)$. $T_{\phi}$ gives an operator from $\ell_1^{(n)}$ to $A_{\phi}^2(D)$. 

\begin{lemma}[see the Appendix of \cite{BG3}]\label{lem-basis}
    There exist a subset $\mathbf{E}$ of $\mathbb{N}^n$, a sequence $\xi[\alpha]=(\xi[\alpha]_{\gamma})_{\gamma}\in\ell_0^{(n)}$, and an orthonormal basis $\{\sigma_{\alpha}\}_{\alpha\in\mathbf{E}}$ of $A_{\phi}^2(D)$, such that for any $\alpha\in\mathbf{E}$,
    
    (1) $\sigma_{\alpha}=T_{\varphi}(\xi[\alpha])$;

    (2) $\max\big\{\gamma\in\mathbb{N}^n : \xi[\alpha]_{\gamma}\neq 0\big\}=\alpha$;

    (3) $\min\big\{\gamma\in\mathbb{N}^n : D^{\gamma}(\sigma_{\alpha})(o)\neq 0\big\}=\alpha$,

where the maximum and minimum in (2) and (3) are w.r.t. `$\prec$'.
\end{lemma}

Here in fact, $\mathbf{E}$ is the set containing all $\alpha\in\mathbb{N}^n$ such that there exists an $f\in A_{\phi}^2(D)$ with $D^{\alpha}f(o)\neq 0$ and $D^{\beta}f(o)=0$ for all $\beta\prec\alpha$. For example, if $D$ is a bounded domain and $\mathcal{I}(\phi)_o=\mathcal{O}_o$, then $\mathbf{E}=\mathbb{N}^n$.

\section{Proof of Theorem \ref{thm-BFID.eq.CFID}}

Now we prove Theorem \ref{thm-BFID.eq.CFID}. For any $\xi\in\ell_0^{(n)}$, we denote
\[\mathrm{ord}_o(\xi):=\max\big\{k\in\mathbb{N} : \exists\alpha\in\mathbb{N}^n, \ \text{s.t.} \ |\alpha|=k, \ \xi_{\alpha}\neq 0\big\},\]
which coincides the order of $\xi$ seen as an element in $\mathbb{C}[z_1,\ldots,z_n]$. We first prove a special case of Theorem \ref{thm-BFID.eq.CFID}.

\begin{lemma}\label{lem-mj0.subset.I}
    Let $D$ be a bounded domain in $\mathbb{C}^n$, $I$ a proper ideal of $\mathcal{O}_o$ and $(F,o)\in\mathcal{O}_o$. Assume that $\mathfrak{m}^{j_0}\subseteq I$ for some $j_0\in\mathbb{N}_+$, where $\mathfrak{m}$ is the maximal ideal of $\mathcal{O}_o$. Then there exists $\eta\in \ell_I\cap\ell_0^{(n)}$ with $\mathrm{ord}_o(\eta)<j_0$ such that
    \[\frac{|(\eta\cdot F)(o)|^2}{K_{\eta,D}(o)}=C_{F,I}(D).\]
\end{lemma}

\begin{proof}
    The case $(F,o)\in I$ is trivial, where we can choose $\eta=(1,0,\ldots,0,\ldots)$ for example. In the following, we assume $(F,o)\notin I$.

    First, we can actually assume that $F\in A^2(D)$. For the reason, if we take the order$\le j_0$ part of the Taylor expansion of $F$ at $o$, denoted by $\tilde{F}$, then $(F-\tilde{F},o)\in\mathfrak{m}^{j_0}\subseteq I$, and $\tilde{F}\in A^2(D)$ since it is a polynomial and $D$ is bounded. Now for any $\xi\in \ell_I$, it follows that $(\xi\cdot F)(o)=(\xi\cdot \tilde{F})(o)$. Replacing $F$ by $\tilde{F}$, it is enough to assume $F\in A^2(D)$.

Since $F\in A^2(D)$ with $(F,o)\notin I$, $C_{F,I}(D)\le\|F\|_D^2<+\infty$, then there exists a unique $G\in A^2(D,I)^{\perp}\setminus\{0\}$ such that $\|G\|^2_D=C_{F,I}(D)$ and $(G-F,o)\in I$. We can write that
\[G=\sum_{\alpha\in\mathbb{N}^n}a_{\alpha}\sigma_{\alpha}, \ a_{\alpha}\in\mathbb{C},\]
where $\{\sigma_{\alpha}\}_{\alpha\in\mathbb{N}^n}$ is the orthonormal basis of $A^2(D)$ in Lemma \ref{lem-basis}. Note that $D^{\beta}(\sigma_{\alpha})(o)=0$ for any $\beta\prec\alpha$, then $\sigma_{\alpha}\in A^2(D,\mathfrak{m}^{j_0})\subseteq A^2(D,I)$ for any $\alpha$ with $|\alpha|\ge j_0$. Thus, $\langle G, \sigma_{\alpha}\rangle_D=0$ if $|\alpha|\ge j_0$, where $\langle\cdot, \cdot\rangle_D$ is the inner product on $A^2(D)$. It follows that we have
\[G=\sum_{|\alpha|\le j_0-1}a_{\alpha}\sigma_{\alpha}, \ b_{\alpha}\in\mathbb{C},\]
which is a finite summation. Set
\[\eta=\sum_{|\alpha|\le j_0-1}\overline{a_{\alpha}}\cdot\xi[\alpha]\in\ell_0^{(n)},\]
where $\xi[\alpha]\in\ell^{(n)}_0$ satisfying $T(\xi[\alpha])=\sigma_{\alpha}$ such as which in Lemma \ref{lem-basis}. Then $\mathrm{ord}_o(\eta)<j_0$, and $T(\eta)=G$. In addition, we have
\[K_{\eta,D}(o)=\sup_{f\in A^2(D)}\frac{|(\eta\cdot f)(o)|^2}{\|f\|_D^2}=\sup_{f\in A^2(D)}\frac{|\langle f, G\rangle_D|^2}{\|f\|_D^2}=\|G\|_D^2,\]
yielding that
\[\frac{|(\eta\cdot F)(o)|^2}{K_{\eta, D}(o)}=\frac{|\langle F, G\rangle_D|^2}{\|G\|_D^2}=\frac{\Big|\|G\|_D^2+\langle F-G, G\rangle_D\Big|^2}{\|G\|_D^2}=\|G\|_D^2=C_{F,I}(D).\]

Now, we only need to verify $\eta\in\ell_{I}$. For any $(f,o)\in I$, by Taylor's expansion, for any $j\ge j_0$, we can write that
\[f=f_j+\tilde{f}_j,\]
where $f_j\in\mathbb{C}[z_1,\cdots,z_n]$ with $\mathrm{deg\ } f_j\le j$, and $\tilde{f}_j\in\mathfrak{m}^{j}\subseteq I$.  Since $D$ is bounded, $f_j\in A^2(D)$. We also have $(f_j,o)=(f-\tilde{f}_j,o)\in I$, thus $f_j\in A^2(D,I)$. Then by that $T(\eta)=G\in A^2(D,I)^{\perp}$, it holds that
\[(\eta\cdot f_j)(o)=\langle f_j, G\rangle_D=0.\]
Since $\tilde{f}_j\in\mathfrak{m}^{j}$ and $\mathrm{ord}_o(\eta)<j_0\le j$, we have $(\eta\cdot \tilde{f}_j)(o)=0$. Thus, $(\eta\cdot f)(o)=0$, for all $(f,o)\in I$. The proof of $\eta\in\ell_I$ is completed, so the desired result holds.
\end{proof}

Using Krull's lemma, we give an approximation result for the minimal $L^2$ integrals.

\begin{lemma}\label{lem-I+mk}
    Let $D$ be a bounded domain in $\mathbb{C}^n$, $I$ a proper ideal of $\mathcal{O}_o$, and $(F,o)\in \mathcal{O}_o$. Then
    \begin{equation}\label{eq-lim.CFI+mjD.eq.CFID}
        \lim_{j\to\infty}C_{F,I+\mathfrak{m}^k}(D)=C_{F,I}(D).
    \end{equation}
\end{lemma}

\begin{proof}
Clearly, $C_{F,I+\mathfrak{m}^k}(D)$ is an increasing sequence in $k$, and $C_{F,I+\mathfrak{m}^k}(D)\le C_{F,I}(D)$ for any $k$. Without loss of generality, we assume $\lim_{k\to\infty}C_{F,I+\mathfrak{m}^k}(D)<+\infty$. Then there exists $G_k\in A^2(D)$ with
\[(G_k-F,o)\in I+\mathfrak{m}^k, \ \|G_k\|_D^2=C_{F,I+\mathfrak{m}^k}(D), \ \forall k\in\mathbb{N}_+.\]
Since $\lim_{k\to\infty}C_{F,I+\mathfrak{m}^k}(D)<+\infty$, by Montel's theorem, we can extract a subsequence of $G_k$ compactly convergent to some $G\in A^2(D)$. Fatou's lemma shows $\|G\|_D^2\le\lim_{k\to\infty}C_{F,I+\mathfrak{m}^k}(D)$. We also have $(G-F,o)\in I+\mathfrak{m}^k$ for any $k$ (see \cite[Chapter 2.2.3]{GR}). Krull's lemma indicates
\[\bigcap_{k\ge 1}(I+\mathfrak{m}^k)=I.\]
Then $(G-F,o)\in I$, which implies $\|G\|_D^2\ge C_{F,I}(D)$. This confirms (\ref{eq-lim.CFI+mjD.eq.CFID}). 
\end{proof}

For an unnecessarily bounded domain, we will use the following lemma to approximate the minimal $L^2$ integral.
\begin{lemma}\label{lem-bounded.approx.unbounded}
    Let $D$ be a domain in $\mathbb{C}^n$ containing $o$, $I$ a proper ideal of $\mathcal{O}_o$, and $(F,o)\in \mathcal{O}_o$. Let $\{D_i\}_{i=1}^{\infty}$ be a sequence of domains such that $D_i\subseteq D_{i+1}$, $\bigcup_{i\ge 1}D_i=D$, and $o\in D_1$. Then
\[\lim_{i\to \infty}C_{F,I}(D_i)=C_{F,I}(D).\]
\end{lemma}

\begin{proof}
    Note that $C_{F,I}(D_i)$ is increasing in $i$ and $C_{F,I}(D_i)\le C_{F,I}(D)$ for any $i$, then without loss of generality, we assume $\lim_{i\to \infty}C_{F,I}(D_i)<+\infty$. For any $i$, let $h_i\in A^2(D_i)$ with $(h_i-F,o)\in I$ and $\|h_i\|_{D_i}^2=C_{F,I}(D_i)$. Then Montel's theorem implies that there exists a subsequence of $\{h_i\}$ which converges to some $h\in\mathcal{O}(D)$ on any compact subset of $D$. According to Fatou's lemma, we have $\|h\|_D^2\le\lim_{i\to \infty}C_{F,I}(D_i)$. We also have $(h-F,o)\in I$, which implies $\|h\|_D^2\ge C_{F,I}(D)$. Then the proof is completed.
\end{proof}

Now we give the proof of Theorem \ref{thm-BFID.eq.CFID}.

\begin{proof}[\textbf{Proof of Theorem \ref{thm-BFID.eq.CFID}}]
We only need to prove $B^{\circ}(F,I,D)\ge C_{F,I}(D)$, where $(F,o)\notin I$.

First, we assume $D$ is bounded. For the ideal $I$ of $\mathcal{O}_o$, since $\mathfrak{m}^k\subset I+\mathfrak{m}^k$, Lemma \ref{lem-mj0.subset.I} implies
\[B^{\circ}(F,I+\mathfrak{m}^k,D)=C_{F,I+\mathfrak{m}^k}(D), \ \forall k\in\mathbb{N}_+,\]
yielding that
\begin{equation}\label{eq-B.>.C.-1/j}
    B^{\circ}(F,I,D)\ge B^{\circ}(F,I+\mathfrak{m}^k,D)=C_{F,I+\mathfrak{m}^k}(D), \ \forall k\in\mathbb{N}_+.
\end{equation}
Combining with Lemma \ref{lem-I+mk}, we get
\[B^{\circ}(F,I,D)\ge C_{F,I}(D).\]

If $D$ is an unbounded domain, we choose an increasing sequence of bounded domains $\{D_i\}_{i=1}^{\infty}$ such that $\bigcup_{i\ge 1}D_i=D$, and $o\in D_1$, then Lemma \ref{lem-bounded.approx.unbounded} implies
\[B^{\circ}(F,I,D)\ge\limsup_{i\to \infty} B^{\circ}(F,I,D_i)=\lim_{i\to \infty} C_{F,I}(D_i)=C_{F,I}(D).\]

At last, the case $(F,o)\in I$ is trivial.

The proof of Theorem \ref{thm-BFID.eq.CFID} is completed.
\end{proof}

We give some remarks here.

\begin{remark}\label{rem-generalized.separation}
    We can reprove Proposition \ref{prop-separation} by Theorem \ref{thm-BFID.eq.CFID} without using the separation theorem in functional analysis theory:
    
    \emph{for any proper ideal $I$ of $\mathcal{O}_o$, and any $(F,o)\in\mathcal{O}_o$, there exists some $\xi\in\ell_I\cap\ell^{(n)}_0$, such that $(\xi\cdot F)(o)\neq 0$.}
    
    This is due to $B^{\circ}(F,I,D)=C_{F,I}(D)\in (0,+\infty)$ for a domain $D$ containing $o$ and $F\in A^2(D)$.
\end{remark}

\begin{remark}\label{rem-B.circ}
    From the proof of Theorem \ref{thm-BFID.eq.CFID}, we can see
    \[C_{F,I}(D)=\sup_{k\ge 1}\sup_{\xi\in\ell_{I+\mathfrak{m}^k}}\frac{|(\xi\cdot F)(o)|^2}{K_{\xi,D}(o)}=\sup_{k\ge 1}\min_{\xi\in\ell_{I+\mathfrak{m}^k}}\frac{|(\xi\cdot F)(o)|^2}{K_{\xi,D}(o)},\]
    where $\ell_{I+\mathfrak{m}^k}$ can be seen as a finite-dimensional linear subspace of $\mathbb{C}[z_1,\ldots,z_n]$.
\end{remark}

\section{Proofs of Corollary \ref{cor-dense} and Corollary \ref{cor-jumping.number.min.xi.cse}}

We prove Corollary \ref{cor-dense} and Corollary \ref{cor-jumping.number.min.xi.cse} in this section.

\begin{proof}[\textbf{Proof of Corollary \ref{cor-dense}}]
It is clear that $T(\ell_I)\subseteq A^2(D,I)^{\perp}$.

Now fix any $F\in A^2(D,I)^{\perp}\setminus\{0\}$. Then by the definition, $C_{F,I}(D)=\|F\|_D^2$. According to Theorem \ref{thm-BFID.eq.CFID}, we can find a sequence $\{\xi_j\}_{j\in\mathbb{N}_+}\subset\ell_I\cap\ell_0^{(n)}$ such that
\[\lim_{j\to+\infty}\frac{|(\xi_j\cdot F)(o)|^2}{K_{\xi_j,D}(o)}=C_{F,I}(D)=\|F\|_D^2.\]
Denote $T(\xi_j)=g_j\in A^2(D,I)^{\perp}$. It follows that
\[K_{\xi_j,D}(o)=\sup_{f\in A^2(D)}\frac{|\langle f, g_j\rangle_D|^2}{\|f\|_D^2}=\|g_j\|_D^2, \ \forall j\in\mathbb{N}_+,\]
which verifies that
\[\|F\|_D^2=\lim_{j\to+\infty}\frac{|\langle F, g_j \rangle_D|^2}{K_{\xi_j,D}(o)}=\lim_{j\to +\infty}\frac{|\langle F, g_j \rangle_D|^2}{\|g_j\|_D^2}.\]
Set
\[\eta_j:=e^{\sqrt{-1}\theta_j}\frac{\|F\|_D}{\|g_j\|_D}\cdot \xi_j\in\ell_I\cap\ell_0^{(n)}, \ \forall j\in\mathbb{N}_+,\]
and $G_j=T(\eta_j)\in A^2(D,I)^{\perp}$, where $\theta_j\in \mathbb{R}$ such that $\langle F,G_j\rangle_D\in\mathbb{R}_{\ge 0}$. Then $\|G_j\|_D^2=\|F\|_D^2$, and 
\[\lim_{j\to +\infty}\langle F,G_j\rangle_D=\lim_{j\to +\infty}\mathrm{Re}\langle F,G_j\rangle_D=\|F\|_D^2.\]
Now we deduce that
\begin{flalign*}
    \begin{split}
        \|F-G_j\|_D^2=\|F\|_D^2+\|G_j\|_D^2-2\mathrm{Re}\langle F, G_j\rangle_D\to 0, \ j\to +\infty,
    \end{split}
\end{flalign*}
which implies $T(\eta_j)=G_j\to F$ as $j\to +\infty$ in the topology of $A^2(D)$. 

The proof is completed.
\end{proof}

\begin{proof}[\textbf{Proof of Corollary \ref{cor-jumping.number.min.xi.cse}}]
    It is proved in \cite{BGY22} that $c_o^F(\varphi)\le \gamma_{\xi}(\varphi)$ for any $\xi\in\ell_0^{(n)}$ with $(\xi\cdot F)(o)\neq 0$ (one can deduce this from the equality (\ref{eq-gamma.xi.varphi.equals.inf})).
    
    In addition, for any $\xi\in\ell_{\mathcal{I}_+(c_o^F(\varphi)\varphi)_o}$ with $(\xi\cdot F)(o)\neq 0$, by (\ref{eq-gamma.xi.varphi.equals.inf}), it holds that $\gamma_{\xi}(\varphi)=c_o^F(\varphi)$. Since $(F,o)\notin \mathcal{I}_+(c_o^F(\varphi)\varphi)_o$, according to Remark \ref{rem-generalized.separation}, there exists some $\eta\in\ell_0^{(n)}\cap\ell_{\mathcal{I}_+(c_o^F(\varphi)\varphi)_o}$ satisfying $(\eta\cdot F)(o)\neq 0$, which implies $\gamma_{\eta}(\varphi)=c_o^F(\varphi)$. Thus, we get
    \[c_o^F(\varphi)=\min_{\xi\in\ell_0^{(n)}, \ (\xi\cdot F)(o)\neq 0}\gamma_{\xi}(\varphi).\]
\end{proof}

\section{A weighted version}

A weighted version of generalized Bergman kernels was introduced in \cite{BG2}, and used to obtain a (sharp) effectiveness result for the $L^p$ strong openness property, which reproved some results in \cite{GY23}.

For the purpose of comparing the results in \cite{BG2} and \cite{GY23}, we demonstrate a weighted version of Theorem \ref{thm-BFID.eq.CFID} in this section. Let $D$ be a domain in $\mathbb{C}^n$ containing $o$, and $\psi$ a plurisubharmonic function on $D$. Denote by 
\[A_{\psi}^2(D):=\Big\{f\in\mathcal{O}(D): \int_D|f|^2e^{-\psi}<+\infty\Big\}\]
the weighted Bergman space and $\|f\|_{D,\psi}:=\left(\int_D|f|^2e^{-\psi}\right)^{1/2}$ the norm on the space.

For any $\xi\in\ell_1^{(n)}$, the \emph{weighted Bergman kernel with respect to $\xi$} on $D$ is denoted by (see \cite{BG2}):
\[K^{\psi}_{\xi,D}(z):=\sup\big\{|(\xi\cdot f)(z)|^2\colon f\in A^2_{\psi}(D), \ \|f\|_{D,\psi}^2\le 1\big\}, \ \forall z\in D.\]
For any proper ideal $I$ of $\mathcal{O}_o$ satisfying $I\not\supseteq\mathcal{I}(\psi)_o$ and $(F,o)\in \mathcal{O}_o$, set
\[B_{\psi}(F,I,D):=\sup_{\xi\in\ell_I}\frac{|(\xi\cdot F)(o)|^2}{K^{\psi}_{\xi,D}(o)},\]
and
\[B^{\circ}_{\psi}(F,I,D):=\sup_{\xi\in\ell_I\cap\ell_0^{(n)}}\frac{|(\xi\cdot F)(o)|^2}{K^{\psi}_{\xi,D}(o)}.\]
In addition, recall the minimal $L^2$ integral:
\[C_{F,\psi,I}(D):=\inf\big\{\|\tilde{F}\|^2_{D,\psi}\colon \tilde{F}\in A^2_{\psi}(D), \ (\tilde{F}-F,o)\in I\big\}.\]

One can establish the following weighted version of Theorem \ref{thm-BFID.eq.CFID}, where the detail of the proof will be omitted here.
\begin{theorem}
    The following equality holds:
    \[B^{\circ}_{\psi}(F,I,D)=B_{\psi}(F,I,D)=C_{F,\psi,I}(D).\]
\end{theorem}

\begin{proof}[Sketch of the proof]
    Let $\psi_j:=\max\{\psi,-j\}$ for each $j\in\mathbb{N}_+$. Since each $\psi_j$ does not have singularity, we can repeat the process of the proof of Theorem \ref{thm-BFID.eq.CFID}, and obtain
    \[B^{\circ}_{\psi_j}(F,I,D)=B_{\psi_j}(F,I,D)=C_{F,\psi_j,I}(D).\]
    Then by approximation, the same equality for $\psi$ can be verified. 
\end{proof}

\vspace{.1in} {\em Acknowledgements}. We thank Dr. Zheng Yuan and Dr. Zhitong Mi for checking this paper. The second named author was supported by National Key R\&D Program of China 2021YFA1003100, NSFC-11825101 and NSFC-12425101.

\end{document}